\documentclass[10pt,conference]{IEEEtran}
\usepackage[ansinew]{inputenc}
\usepackage[dvips]{graphicx}
\usepackage[T1]{fontenc}
\usepackage{amsmath,enumerate,amssymb,hhline,verbatim}

\begin{document}

\title{Construction of MIMO MAC Codes Achieving the Pigeon Hole Bound}
\author{\IEEEauthorblockN{Toni Ernvall and Roope Vehkalahti}
\IEEEauthorblockA{Turku Center for Computer Science and Department of Mathematics, University of Turku, Finland\\
e-mails: \{tmernv, roiive\}@utu.fi\\
}
}

\maketitle

\newtheorem{definition}{Definition}[section]
\newtheorem{thm}{Theorem}[section]
\newtheorem{proposition}[thm]{Proposition}
\newtheorem{lemma}[thm]{Lemma}
\newtheorem{corollary}[thm]{Corollary}
\newtheorem{exam}{Example}[section]
\newtheorem{conj}{Conjecture}
\newtheorem{remark}{Remark}[section]

\newcommand{\La}{\mathbf{L}}
\newcommand{\h}{{\mathbf h}}
\newcommand{\Z}{{\mathbf Z}}
\newcommand{\R}{{\mathbf R}}
\newcommand{\C}{{\mathbf C}}
\newcommand{\D}{{\mathcal D}}
\newcommand{\F}{{\mathbf F}}
\newcommand{\HH}{{\mathbf H}}
\newcommand{\OO}{{\mathcal O}}
\newcommand{\G}{{\mathcal G}}
\newcommand{\A}{{\mathcal A}}
\newcommand{\B}{{\mathcal B}}
\newcommand{\I}{{\mathcal I}}
\newcommand{\E}{{\mathcal E}}
\newcommand{\PP}{{\mathcal P}}
\newcommand{\Q}{{\mathbf Q}}
\newcommand{\M}{{\mathcal M}}
\newcommand{\separ}{\,\vert\,}
\newcommand{\abs}[1]{\vert #1 \vert}

\begin{abstract}
This paper  provides a general construction method for multiple-input multiple-output multiple access channel codes (MIMO MAC codes) that have
so called generalized full rank property. The achieved constructions  give a positive  answer to  the question whether it is generally  possible to reach the so called pigeon hole bound, that is an upper bound for the decay of determinants of MIMO-MAC channel codes.
\end{abstract}
\section{Introduction}
In MIMO MAC  the  DMT  optimality criteria can be given by splitting the whole error probability space to separate error events and then giving a criteria for each event separately. In the case where each user is in error, the criteria   given in \cite{CoGaBo}  closely resembles that of the classical NVD condition used in single user case. This criteria would give us a natural goal if we would like to build DMT achieving codes.
Unfortunately  in \cite{LuHoVeLa} it was proved that this criteria is too tight  in the case where the codes of the single users  are lattice  space-time codes. The  \emph{pigeon hole bound} in \cite{LuHoVeLa} proves that, irrespective of the code design, the determinants of the
overall code matrices will decay with  a polynomial  speed.

In  \cite{decaypaperi} it was proved that the two user single antenna code (BB-code)  given in \cite{bb} do achieve the pigeon hole bound. In this paper we are giving a wast  generalization of  BB-code  to general MAC-MIMO. The  codes we will build  fulfill the so called \emph{generalized rank criteria} and  do reach the pigeon hole bound.

The single user codes we are using are based on the \emph{multi-block} codes from division algebras \cite{Lu}.
This  approach  has been taken in several recent papers on MIMO-MAC.  However, in these papers the full rank criteria has been achieved by using either transcendental  elements \cite{bb} or algebraic elements with high degree \cite{matriisi}. Both of these methods make it extremely difficult to measure the decay of the codes and likely  lead to bad decay.

Instead of the usual algebraic independence strategy  we will use \emph{valuation theory} to achieve the full rank condition. This  technical tool allows us  to use algebraic elements with low degree. By applying Galois theoretic method of Lu et al. \cite{LuHoVeLa} and  methods from Diophantine approximation, originally introduced by  Lahtonen  et al. in \cite{decaypaperi}, we will prove  that our codes achieve   good decay and in particular reach  the pigeon hole bound.

As a warm up we discuss, when it is possible to achieve the generalized full rank condition for two users, each having a single transmit antenna. In  particular these results clarify the  construction of BB-code in \cite{bb}. Then in Section  \ref{construction} we will give our general construction and state that such an construction always exists and that each code does fulfill the generalized  rank criteria. In Section \ref{decayanalysis} we  analyze the decay of the constructed codes.  In particular  Theorem  \ref{inertdecay3} gives a general lower bound and Corollary  \ref{pigeon2} shows that our codes achieves the pigeon hole bound.  In the last section we give few explicit examples of our codes.

 \section{Decay function, multiuser code, and other definitions}
In this paper we are considering  MIMO MAC code design in the Rayleigh  flat fading channel.
 We suppose that we have  $U$ users, each having  $n_t$ antennas and that the receiver has $n_r$ antennas and  complete channels state information.
 We also suppose that the fading for each user  stays stable for $Un_t$ time units. Let us refer to the channel matrix of the $i$th  user with $H_i\in M_{n_r\times n_t}$ and let us suppose that each of these have  i.i.d complex Gaussian random variables with zero mean and unit variance as coefficients. In this scenario the base station receives
$$
Y=\sum_{i=1}^{U} H_iX_i +N,
$$
where, $X_i\in M_{n_t\times Un_t}(\C)$,  is the transmitted codeword from the $i$th user, and $N \in \M_{n_r\times n_t U}(\C)$ presents the noise        having i.i.d complex Gaussian random variables as coefficients.

In this scenario  multiuser MIMO signal is  a $Un_t \times Un_t$ matrix where the rows $(j-1)n_t + 1, (j-1)n_t + 2, \dots, (j-1)n_t + n_t$ represent $j$th user's signal ($j=1, \dots, U$).

We  suppose that  each user applies a lattice space-time code $\textbf{L}_j \subseteq \M_{n_t \times Un_t}, j=1, \ldots, U$. We  also assume that each user's lattice is of full rank $r=2Un_{t}^2$, and denote the basis of the  lattice $\textbf{L}_j$ by $B_{j,1}, \ldots, B_{j,r}$. Now the code associated with the $j$th user is a restriction of lattice $\textbf{L}_{j}$
$$
\textbf{L}_{j} (N_{j}) = \left\{ \sum_{i=1}^{r} b_{i} B_{j,i} | b_{i} \in \Z, -N_{j} \leq b_{i} \leq N_{j} \right\},
$$
where $N_{j}$ is a given positive number.

Using these definitions the \emph{$U$-user MIMO MAC code} is $(\textbf{L}_{1} (N_{1}), \textbf{L}_{2} (N_{2}), \dots, \textbf{L}_{U} (N_{U}))$.

For a matrix $M(X_{1}, \ldots, X_{U})=(X_{1}^{T}, X_{2}^{T}, \ldots, X_{U}^{T})^{T}$ where $X_j \in \textbf{L}_{j}(N_{j})$ that $j$th user has sent, we define
$$
D(N_{1}, \ldots, N_{U}) = \min_{X_{j} \in \textbf{L}_{j}(N_{j}) \setminus \{ 0 \} } \left| \det M(X_{1}, \ldots, X_{U}) \right|.
$$
This is called  the \emph{decay function}. In the  special case $N_{1} = \ldots = N_{U} = N$ we write
$$
D(N) = D(N_{1} = N , \ldots, N_{U} = N).
$$

If we have a $U$-user code with decay function $D(N_{1}, \ldots, N_{n})$, such that $D(N_{1}, \ldots, N_{n}) \neq 0$ for all $N_{1}, \ldots, N_{U} \in \Z_{+}$,  we say that the code satisfies \emph{ (generalized) rank criterion}.

\section{2-user code}
In \cite{bb}  Badr and Belfiore introduced  a $2$-user single antenna MAC  code where the matrix coefficients were from the field $\Q(i,\sqrt{5})$. In the following  we give a complete characterization of when their construction method leads to codes with generalized full rank property.
\begin{thm}
\label{2olemassa}
Let $K/\Q$ and $L/K$ be two field extensions of degree $2$ and $a,b,c,d \in L$. Let also $\sigma$ be the non-trivial element in the Galois group $Gal(L/K)$. Define
$$
C = \left\{ \left(
         \begin{array}{cc}
           ax & b \sigma(x) \\
           cy & d \sigma(y) \\
         \end{array}
       \right) | x,y \in \OO_{L}^{*} \right\}.
$$
There exists a matrix in $C$ with zero determinant if and only if
\begin{equation}
\label{kaava1}
\left|
         \begin{array}{cc}
           N(a) & N(b) \\
           N(c) & N(d) \\
         \end{array}
       \right|=0,
\end{equation}
where the function $N=N_{L/K}$ denotes the norm of extension $L/K$.
\end{thm}
\begin{proof}
Assume first that we have a matrix
$$\left(
         \begin{array}{cc}
           ax & b \sigma(x) \\
           cy & d \sigma(y) \\
         \end{array}
       \right)$$
        with zero determinant. This means that $adx\sigma(y)=bc\sigma(x)y$, which gives $N(a)N(d)N(x)N(\sigma(y))=N(b)N(c)N(\sigma(x))N(y)$. Continuing we get  $N(a)N(d)=N(b)N(c)$ \emph{i.e.} $N(a)N(d)-N(b)N(c)=0$.

Assume then that  $N(a)N(d)-N(b)N(c)=0$. If $b$ or $c$ is zero then $N(a)N(d)=0$ \emph{i.e.} $a$ or $d$ is zero and then for all $x,y \in \OO_L$ we have
 $$
 \left|
         \begin{array}{cc}
           ax & b \sigma(x) \\
           cy & d \sigma(y) \\
         \end{array}
       \right| = 0.
 $$
 Otherwise $N(\frac{ad}{bc})=1$. Then by Hilbert 90 we have some $z \in L$ such that $\frac{ad}{bc}=\frac{\sigma(z)}{z}$. Then write $z=\frac{w}{n}$ with $w \in \OO_{L}^{*}$ and $n \in \Z$. This gives $\frac{ad}{bc}=\frac{\sigma(w)}{w}$ \emph{i.e.} $adw-bc\sigma(w)=0$. This means that the determinant of
$$
\left(
         \begin{array}{cc}
           aw & b \sigma(w) \\
           c1 & d \sigma(1) \\
         \end{array}
       \right)
$$
is zero.
\end{proof}

\section{Multi access code with several transmission antennas}\label{construction}

\subsection{Construction of $U$-user code for $n_t$ transmission antennas}

From now on we concentrate  on the  scenario where we have $U \in \Z_+$ users and each user has $n_t \in \Z_+$ transmission antennas.  Throughout this section we assume $K$ to be an imaginary quadratic extension of $\Q$ with class number $1$, $L$ is a cyclic Galois extension of $K$ of degree $Un_t$, such that $L=K(\alpha)$ with $\alpha \in \R$, $\sigma$ a generating element in $Gal(L/K)$ and $p \in \OO_{K}$ an inert prime in $L/K$. We also define $\tau = \sigma^{U}$ and $F$ to be the fixed field of $\tau$. So we have $[L:F]=n_t$, $[F:K]=U$, $Gal(L/F)=<\tau>$, and $Gal(F/K)=<\sigma_{F}>$ where $\sigma_{F}$ is a restriction of $\sigma$ in $F$. Let  $v=v_p$ be the   $p$-adic valuation  of the field L. In this section, when we say that $L/K$, $p$, and $\sigma$ are suitable we mean that they are as above.

Due to the space limitations, we skip the proof of the following proposition.
\begin{proposition}\label{existence}
For every complex quadratic field $K$, having class number 1, and for any $U$ and for any $n_t$ we have   a suitable degree $n_t U$ extension $L/K$,  prime $p\in \OO_K$ and automorphism $\sigma \in Gal(L/K)$.
\end{proposition}

Now we are ready for our construction.
Let $x_j \in \OO_{L}$ for all $j=1, \dots ,n_t$. Define $M=M(x_1, x_2, \dots, x_{n_t})$ to be
$$
\left(
         \begin{array}{ccccc}
           x_1   &     p \tau(x_{n_t}) & p \tau^2(x_{n_t-1}) & \ldots & p \tau^{n_t-1}(x_2) \\
           x_2   &       \tau(x_1) & p \tau^2(x_{n_t}) & \ldots & p \tau^{n_t-1}(x_3) \\
           x_3   &       \tau(x_2) &   \tau^2(x_1) & \ldots & p \tau^{n-1}(x_4) \\
           \vdots & \vdots         & \vdots          &        & \vdots              \\
           x_{n_t-1}   &   \tau(x_{n_t-2}) &   \tau^2(x_{n_t-3}) & \ldots & p \tau^{n_t-1}(x_{n_t}) \\
           x_{n_t}   &       \tau(x_{n_t-1}) &   \tau^2(x_{n_t-2}) & \ldots &   \tau^{n_t-1}(x_1) \\
         \end{array}
       \right)
$$
\emph{i.e.} a matrix representation of an element in the cyclic  algebra $(L/F,\tau,p)$.
The following lemma proofs that  $(L/F,\tau,p)$ is a division algebra.

\begin{lemma}
\label{determinanttivaluaatio}
Let $x_j \in \OO_{L}$ for all $j=1, \dots , n_t$ such that $x_l \neq 0$ for some $l$ and $\min(v(x_1), \dots, v(x_{n_t}))=0$. Then we have
$$
\det(M(x_1, x_2, \dots, x_{n_t})) \neq 0
$$
and
$$
v(\det(M(x_1, x_2, \dots, x_{n_t}))) \leq n_t - 1.
$$
\end{lemma}
\begin{proof}
Write $M=M(x_1, x_2, \dots, x_{n_t})$ and $N=N_{L/F}$. Assume first that $v(x_1)=0$. Then the determinant is $N(x_1) + py$ for some $y \in \OO_L$ and hence we have $v(\det(M))=\min(v(N(x_1)),v(py))=0$. Assume then that $v(x_1),v(x_2), \dots, v(x_{l-1}) > 0$ and $v(x_l)=0$ with $1 < l \leq n_t$. Notice that in this case all the other elements $a$ of matrix $M$ than those in the left lower corner  block of side length $n_t -l+1$ have $v(a)>0$. Either they have coefficient $p$ or they are automorphic images of elements $x_1 , x_2 , \dots , x_{l-1}$. Now $\det(M) = \pm p^{l-1}N(x_l) + p^{l}z$ for some $z \in \OO_L$ since all the other terms except $\pm p^{l-1}N(x_l)$ have at most $n_t -l$ factors from this left lower corner and hence at least $n_t-(n_t-l)=l$ terms have factor $p$. This gives that $v(\det(M))=\min(v(p^{l-1}N(x_l)),v(p^{l}z))=l-1 \leq n_t -1$.
\end{proof}

\begin{definition}
Define $M_j=M(x_{j,1}, x_{j,2}, \dots, x_{j,n_t})$ for all $j=1, \dots ,U$. In our multi access system the code $\mathcal{C}_j$ of $j$th user consists of $n_t \times Un_t$ matrices $B_j=$
$$
 \left( M_j , \sigma(M_j) , \sigma^2(M_j) , \dots , p^{-k} \sigma^{j-1}(M_j) , \dots , \sigma^{U-1}(M_j) \right)
$$
where $k$ is any rational integer strictly greater than $\frac{U(n_t-1)}{2}$ and $x_{j,l} \neq 0$ for some $l$. Here  $k$ is same for all users. The the  whole code $C_{U,n_t}$ consists of matrices

$$
A = \left(
         \begin{array}{c}
           B_1   \\
           B_2    \\
           \vdots              \\
           B_U \\
         \end{array}
       \right),
$$
where $B_j \in \mathcal{C}_j$ for all $i=1, \dots, U$. This means that the matrices $A \in C_{U,n_t}$ have form
$$
\left(
         \begin{array}{cccc}
           p^{-k} M_1   &  \sigma(M_1) & \ldots &  \sigma^{U-1}(M_1) \\
           M_2   & p^{-k} \sigma(M_2) & \ldots &  \sigma^{U-1}(M_2) \\
           \vdots & \vdots                  &        & \vdots              \\
           M_U   &  \sigma(M_U) & \ldots & p^{-k} \sigma^{U-1}(M_U) \\
         \end{array}
       \right).
$$

We will skip the proof of the following proposition, stating that each  of the single user codes satisfies the NVD condition and are therefore DMT optimal.
\begin{proposition}
Let $C_{U,n_t} \in \mathfrak{C}_{U,n_t}$ and $\mathcal{C}_j$ be the $j$th users code in the system $C_{U,n_t}$ for some $j \in {1, \dots, U}$. Then the code $\mathcal{C}_j$ is a $2Un_t^2$-dimensional lattice code    with the NVD property.
\end{proposition}

The code depends on how we did choose $L/K$, $p$, $\sigma$, and $k$, so to be precise, we can also refer to $C_{U,n_t}$ with $C_{U,n_t}(L/K,p,\sigma,k)$. Let us call the family of all such codes $C_{U,n_t}(L/K,p,\sigma,k)$ (\emph{i.e.} codes constructed with any suitable $L/K$, $p$, $\sigma$, and $k$) by $\mathfrak{C}_{U,n_t}$. That is
$$
\mathfrak{C}_{U,n_t} = \bigcup_{L/K,p,\sigma,k} \{ C_{U,n_t}(L/K,p,\sigma,k) \}
$$
where $L/K$, $p$, $\sigma$, and $k$ are any suitable ones.
\end{definition}

According to Proposition \ref{existence}  we can always find suitable $L/K$, $p$, $\sigma$, and $k$ for any $U \in \Z_+$ and $n_t \in \Z_+$. We therefore  have the following theorem.
\begin{thm}
For any choice of $U \in \Z_+$ and $n_t \in \Z_+$ we have $\mathfrak{C}_{U,n} \neq \emptyset$.
\end{thm}

Note that the code $C_{U,1}=C_{U,1}(L/K,p,\sigma,1) \in \mathfrak{C}_{U,1}$, a code for $U$ users each having one transmission antenna, consists of matrices of form

$$
\left(
         \begin{array}{ccccc}
           p^{-1}x_1   &         \sigma(x_1) &       \sigma^2(x_1) & \ldots &         \sigma^{U-1}(x_1) \\
                 x_2   & p^{-1}  \sigma(x_2) &       \sigma^2(x_2) & \ldots &         \sigma^{U-1}(x_2) \\
                 x_3   &         \sigma(x_3) & p^{-1}\sigma^2(x_3) & \ldots &         \sigma^{U-1}(x_3) \\
           \vdots      & \vdots              & \vdots              &        & \vdots                    \\
                 x_U   &         \sigma(x_U) &       \sigma^2(x_U) & \ldots & p^{-1}  \sigma^{U-1}(x_U) \\
         \end{array}
       \right).
$$

Note also that the code $C_{1,n_t}=C_{1,n_t}(L/K,p,\sigma,k) \in \mathfrak{C}_{1,n_t}$ is a usual single user  code multiplied by $p^{-k}$.

\begin{thm}
Let $C_{U,n_t} \in \mathfrak{C}_{U,n_t}$. The code $C_{U,n_t}$ is a full rate code and   satisfies the generalized rank criterion.
\end{thm}
\begin{proof}
Let $A \in C_{U,n_t}=C_{U,n_t}(L/K,p,\sigma,k)$. We may assume that $\min(v(x_{j,1}), \dots, v(x_{j,n_t}))=0$, for all $j=1, \dots, U$, because otherwise we can divide extra $p$'s off. That does not have any impact on whether $\det(A)=0$ or not. The determinant of $A$ is
$$
p^{-kUn_t} \prod_{l=1}^{U} \det(\sigma^{l-1}(M_l)) + y
$$
where $v(y) \geq -k(Un_t-2)$. We know that $v(\sigma^{l-1}(\det(M_l)))=v(\det(M_l))$ because $p$ is from $K$, \emph{i.e.} from the fixed field of $\sigma$, and $\det(M_l) \neq 0$ for all $l$. Therefore
$$
v(p^{-kUn_t} \prod_{l=1}^{U} \det(\sigma^{l-1}(M_l))) = -kUn + \sum_{l=1}^{U} v(\det(M_l))
$$
that is less or equal than $-kUn_t + U(n_t-1) = U(n_t-1-kn_t)$ by \ref{determinanttivaluaatio}. But if we would have $\det(A)=0$ then
$$
v(y) = v(p^{-kUn_t} \prod_{l=1}^{U} \det(\sigma^{l-1}(M_l)))
$$
and hence $v(y) \leq U(n_t-1-kn_t)$ implying $-k(Un_t-2) \leq U(n_t-1-kn_t)$. This gives $2k \leq U(n_t-1)$ \emph{i.e.} $k \leq \frac{U(n_t-1)}{2}$ a contradiction.
\end{proof}

\begin{remark}
Using  multiblock codes from division algebras as single user codes in the MIMO MAC scenario has been used before for example in \cite{bb}, \cite{matriisi} and \cite{LuHoVeLa}. In \cite{bb} the full rank condition for codes with $n_t>1$ is achieved by using transcendental elements.
In  \cite{matriisi} the same effect is achieved with algebraic elements of high degree.

\end{remark}
\section{On the decay function of codes in $\mathfrak{C}_{U,n_t}$}\label{decayanalysis}
In this chapter we will prove an asymptotic lower bound for the decay function of codes from $\mathfrak{C}_{U,n_t}$. In \cite{decaypaperi} the authors give a general asymptotic upper bound for a decay function in the case that only one user is properly using the code \emph{i.e.} $N_1$ can be anything but $N_2= \dots =N_U=1$ are restricted. We will see that in this special case our codes have asymptotically the best possible decay.

\begin{lemma}
Let $C_{U,n_t}=C_{U,n_t}(L/K,p,\sigma,k) \in \mathfrak{C}_{U,n_t}$, $A \in C_{U,n_t}$, and let $F$ be the fixed field of $\tau = \sigma^{U}$. Then $\det(A) \in F$.
\end{lemma}
\begin{proof}
As $F$ is the fixed field of $\tau$ it is enough to prove that $\tau(\det(A))=\det(A)$. Write again $M_j=M(x_{j,1}, x_{j,2}, \dots, x_{j,n_t})$ and notice that $\tau(M_j)$, after switching first and last column and first and last row, is
$$
\left(
         \begin{array}{ccccc}
           x_{j,1}   &      \tau(x_{j,n_t}) &  \tau^2(x_{j,n_t-1}) & \ldots &  \tau^{n_t-1}(x_{j,2}) \\
           p x_{j,2}   &       \tau(x_{j,1}) & p \tau^2(x_{j,n_t}) & \ldots & p \tau^{n_t-1}(x_{j,3}) \\
           p x_{j,3}   &       \tau(x_{j,2}) &   \tau^2(x_{j,1}) & \ldots & p \tau^{n_t-1}(x_{j,4}) \\
           \vdots & \vdots         & \vdots          &        & \vdots              \\
           p x_{j,n_t}   &       \tau(x_{j,n_t-1}) &   \tau^2(x_{j,n_t-2}) & \ldots &   \tau^{n_t-1}(x_{j,1}) \\
         \end{array}
       \right).
$$
Here we did an even number of row and column changes so when calling the above matrix $M_{j}'$ we see that $\tau(\det(A))$ is
$$
\left|
         \begin{array}{ccccc}
           p^{-k} M'_1   &  \sigma(M'_1) &  \sigma^2(M'_1) & \ldots &  \sigma^{n-1}(M'_1) \\
           M'_2   & p^{-k} \sigma(M'_2) &  \sigma^2(M'_2) & \ldots &  \sigma^{n-1}(M'_2) \\
           M'_3   &  \sigma(M'_3) & p^{-k} \sigma^2(M'_3) & \ldots &  \sigma^{n-1}(M'_3) \\
           \vdots & \vdots        & \vdots          &        & \vdots              \\
           M'_U   &  \sigma(M'_U) &  \sigma^2(M'_U) & \ldots & p^{-k} \sigma^{n-1}(M'_U) \\
         \end{array}
       \right|.
$$
Let us call the above matrix $A'$. We notice that matrices $A$ and $A'$ are exactly similar apart from the fact that in $M_j$ elements in places $(1,2),(1,3),\dots,(1,n_t)$ have coefficient $p^{-k}$ and elements in places $(2,1),(3,1),\dots,(n_t,1)$ have not coefficient $p^{-k}$ and in $M_{j}'$ these are vice versa. But this does not change the value  of the determinant because in the expansion of the determinant each summand include a product of  same number of  elements from columns that are congruent to $1$ modulo $n_t$ and from rows that are congruent to $1$ modulo $n_t$.
\end{proof}

\begin{thm}\cite{remarks} \label{pigeonmultiantenna}
For any full-rate $U$-user lattice code, each user transmitting with $n_t$ antennas there exists a constant $k>0$ such that
$$
D(N_1=N, N_2=N_3= \dots =N_U=1) \leq \frac{k}{N^{(U-1)n_t}}.
$$
\end{thm}

For the next theorem we need few definitions. Let $p(x)=p_0+p_1x+\dots+p_m x^m \in \Z[x]$ be a polynomial. Then we say that $H(p(x))=\max\{|p_j|\}$ is the height of the polynomial $p(x)$ and for an algebraic number $\alpha$ we define $H(\alpha)=H(\phi_{\alpha})$ where $\phi_{\alpha}$ is the minimal polynomial of $\alpha$.
The next generalization of Liouville's theorem can be found from \cite[p. 31]{shidlovskii}.
\begin{thm}
\label{algebraicapprox}
Let $\alpha \in \R$ be an algebraic number of degree $\kappa$, $H(\alpha) \leq h$, $H(P) \leq H$ and $\deg(P(x))=m \in \Z^{+}$. Then either $P(\alpha)=0$ or
$$
|P(\alpha)| \geq \frac{c^{m}}{H^{\kappa-1}}
$$
with $c=\frac{1}{3^{\kappa-1}h^{\kappa}}$.
\end{thm}

Now we are ready to give a  lower bound for the decay function of our codes. The  proof can be seen as an extension to the analysis given for BB-code in \cite{decaypaperi}.
\begin{thm}
\label{inertdecay3}
For a code $C_{U,n_t} \in \mathfrak{C}_{U,n_t}$ there exists constant $K>0$ such that
$$
D(N_1, N_2, \ldots , N_U) \geq \frac{K}{(N_1 N_2 \ldots N_U)^{(U-1)n_t}}.
$$
Especially
$$
D(N) \geq \frac{K}{N^{U(U-1)n_t}}.
$$
\end{thm}
\begin{proof}
Let $C_{U,n_t}=C_{U,n_t}(L/K,p,\sigma,k)$. Field extension $L/\Q$ has a basis $S_1 \cup S_2$ where $S_1=\{1, \delta, \delta^{2}, \ldots, \delta^{U-1}, \beta, \beta \delta, \beta \delta^{2}, \ldots, \beta \delta^{U-1}\}$ is a basis of $F/\Q$ with $\delta \in \R$, $K=\Q(\beta)$ and $\beta=\sqrt{-w}$ for some positive integer $w$. Notice that if $L=F$ then $S_2 = \emptyset$.

 The ring $\OO_L$ has a $\Z$-basis $\{ \gamma_{1}, \ldots, \gamma_{2Un_t} \}$. Each of these basis elements can be presented as
$$
\gamma_{l}=\sum_{a \in S_1 \cup S_2} s_{l,a} a,
$$
where $s_{l,a} \in \Q$ for all $l = 1, \dots, 2Un_t$ and $a \in S_1 \cup S_2$.

Let $A \in C_{U,n}$ be
$$
\left(
         \begin{array}{ccccc}
           p^{-k} M_1   &  \sigma(M_1) & \ldots &  \sigma^{U-1}(M_1) \\
           M_2   & p^{-k} \sigma(M_2) & \ldots &  \sigma^{U-1}(M_2) \\
           \vdots & \vdots          &        & \vdots              \\
           M_U   &  \sigma(M_U) & \ldots & p^{-k} \sigma^{U-1}(M_U) \\
         \end{array}
       \right)
$$
and $M_j=M(x_{j,1}, x_{j,2}, \dots, x_{j,n_t})$ be
$$
\left(
         \begin{array}{ccccc}
           x_{j,1}   &     p \tau(x_{j,n_t}) & p \tau^2(x_{j,n_t-1}) & \ldots & p \tau^{n_t-1}(x_{j,2}) \\
           x_{j,2}   &       \tau(x_{j,1}) & p \tau^2(x_{j,n_t}) & \ldots & p \tau^{n_t-1}(x_{j,3}) \\
           x_{j,3}   &       \tau(x_{j,2}) &   \tau^2(x_{j,1}) & \ldots & p \tau^{n_t-1}(x_{j,4}) \\
           \vdots & \vdots         & \vdots          &        & \vdots              \\
           x_{j,n_t}   &       \tau(x_{j,n_t-1}) &   \tau^2(x_{j,n_t-2}) & \ldots &   \tau^{n_t-1}(x_{j,1}) \\
         \end{array}
       \right)
$$
as usual.

Now for any $j=0, \dots, Un_t-1$ we have
$$
\sigma^{j}(x_{m,h}) = \sum_{l=1}^{2Un_t} u_{m,h,l} \sigma^{j}(\gamma_l)
$$
where $u_{m,h,l} \in \Z$ and $|u_{m,h,l}| \leq N_{m}$ for all $m$, $h$ and $l$.

Then the determinant $\det(A)$ is a sum consisting of $Un_t!$ elements of form
$$
p^{-f} \prod_{j=0}^{Un_t}  \sigma^{j}(x_{m_j,h_j}) = p^{-f} \prod_{j=0}^{Un_t} ( \sum_{l=1}^{2Un_t} u_{m_j,h_j,l} \sigma^{j}(\gamma_l) )
$$
where $f \leq kUn_t$ and $m_j$ gets exactly $n_t$ times all the values $1, \dots, U$ and $h_j$ gets values from $\{ 1, \dots, n_t \}$.

Now substituting $\gamma_{l}=\sum_{a \in S_1 \cup S_2} s_{l,a} a$ gives that the determinant is a sum consisting of elements of form
$$
p^{-f} \prod_{j=0}^{Un_t} ( \sum_{l=1}^{2Un_t} u_{m_j,h_j,l} \sum_{a \in S_1 \cup S_2} s_{l,a} \sigma^{j}(a) ).
$$

We also write
$$
\sigma^{j}(a)=\sum_{a \in S_1 \cup S_2} t_{j,a} a
$$
where $t_{j,a} \in \Q$ for all $j,a$ and find that $p^{-f} \prod_{j=0}^{Un_t} ( \sum_{l=1}^{2Un_t} u_{m_j,h_j,l} \sum_{a \in S_1 \cup S_2} s_{l,a} \sigma^{j}(a) )$ can be written as a sum of elements of form
$$
K_{1} p^{-f} \sum_{a \in S_1 \cup S_2} u_a a
$$
where $K_{1} \in \Q$ is some constant, $u_a \in \Z$, and $u_a = \OO((N_1 \dots N_U)^{n_t})$.

Writing also $p$ using basis $S_1 \cup S_2$ we see that the whole determinant $\det(A)$ can be written as a sum of elements of form
$$
\sum_{a \in S_1 \cup S_2} u_{a}' a
$$
multiplied by some constant $K_{2}$ and here we have $u_{a}' \in \Z$, and $u_{a}' = \OO((N_1 \dots N_U)^{n_t})$.

On the other hand we know that $\det(A) \in F$ so by uniqueness of basis representation we know that $\det(A)$ is a sum consisting of elements of form
$$
\sum_{a \in S_1} u_{a}' a = \sum_{m=0}^{U-1} u_{\delta^{m}}' \delta^{m} + \beta \sum_{m=0}^{U-1} u_{\delta^{m}\beta}' \delta^{m}
$$
and hence
$$
|\det(A)| = |K_{2}| |\sum_{l=0}^{U-1} H_{l} \delta^{l} + \beta \sum_{l=0}^{U-1} J_{l} \delta^{l}|,
$$
where $H_l , J_l \in \Z$ and $|H_l| , |J_l|$ are of size $\OO((N_1 \cdots N_U)^{n_t})$ for all $l = 0, \dots, U-1$.

Using the fact that $\delta$ is real we get
$$
|\det(A)| \geq \frac{K_{2}}{2} (|\sum_{l=0}^{U-1} H_{l} \delta^{l}| + |\sum_{l=0}^{U-1} J_{l} \delta^{l}|).
$$

Now using \ref{algebraicapprox} and noticing that $\deg(\delta) = U$ we have
$$
|\det(A)| \geq \frac{K}{(N_{1} \cdots N_{U})^{(U-1)n_t}},
$$
where $K$ is some positive constant.

\end{proof}

\begin{corollary}\label{pigeon2}
For a code $C_{U,n_t} \in \mathfrak{C}_{U,n_t}$ there exists constants $k>0$ and $K>0$ such that
$$
\frac{k}{N^{(U-1)n_t}} \leq D(N_1=N, N_2= \ldots =N_U=1) \leq \frac{K}{N^{(U-1)n_t}}.
$$
\end{corollary}

\section{Examples}\label{examples}
Let us now give few examples of our general code constructions.
In table Table 1 we have  collected  some examples of suitable fields $K$ and $L$ and inert primes $p$,  fulfilling the conditions of Proposition \ref{existence}. If $K=\Q(i)$ then   $p_i\OO_K$ refers to the inert prime and if $K=\Q(\sqrt{-3})$ then $p_{\sqrt{-3}}\OO_K$ is inert. The inert primes and fields $L$ are found by looking at totally real subfields of  $\Q(\zeta_m)/\Q$ and then composing  them  with the field $K$.

\begin{table}[h!]\label{table1}\caption{}
$$
\begin{array}{ccccc}
\hline
           [L:K] & L                                                              & p_i & p_{\sqrt{-3}} \\
\hline
           3     & K(\zeta_7+\zeta_{7}^{-1})                                      & 2+i & \sqrt{-3} \\
           4     & K(\zeta_{17}+\zeta_{17}^{4}+\zeta_{17}^{-4}+\zeta_{17}^{-1})   & 2+i & \sqrt{-3}\\
           5     & K(\zeta_{11}+\zeta_{11}^{-1})                                  & 1+i & 2+\sqrt{-3}\\
           6     & K(\zeta_{13}+\zeta_{13}^{-1})                                  & 1+i & 2+\sqrt{-3}\\
           7     & K(\zeta_{29}+\zeta_{29}^{12}+\zeta_{29}^{-12}+\zeta_{29}^{-1}) & 1+i & \sqrt{-3}\\
\hline
\end{array}
$$
\end{table}

\subsection{Actual examples}
We get a code $C_{3,1}=C_{3,1}(\Q(i,\zeta_7+\zeta_{7}^{-1}),2+i,\sigma,1)$ \emph{i.e.} 3-user code with each user having 1 antenna by setting $L=K(\zeta_7+\zeta_{7}^{-1})$, $K=\Q(i)$, $p=2+i$, and $Gal(L/K)=<\sigma>$. Now the actual code consists of matrices
$$
\left(
         \begin{array}{ccc}
          p^{-1} x   &  \sigma(x)       & \sigma^2(x)         \\
           y         & p^{-1} \sigma(y) & \sigma^2(y)         \\
           z         &  \sigma(z)       & p^{-1} \sigma^2(z)  \\
         \end{array}
       \right)
$$
where $x,y,z \in \OO_{L}^{*}$.

We get a code $C_{2,2}=C_{2,2}(\Q(\sqrt{-3},\zeta_{17}+\zeta_{17}^{4}+\zeta_{17}^{-4}+\zeta_{17}^{-1}),\sqrt{-3},\sigma,2)$ \emph{i.e.} 2-user code with each user having 2 antennas by setting $L=K(\zeta_{17}+\zeta_{17}^{4}+\zeta_{17}^{-4}+\zeta_{17}^{-1})$, $K=\Q(\sqrt{-3})$, $p=\sqrt{-3}$, $k=2 > \frac{U(n_t-1)}{2}$, and $Gal(L/K)=<\sigma>$. Now the actual code consists of matrices
$$
\left(
         \begin{array}{cccc}
          p^{-2} x_1   &     p^{-1} \sigma^2(x_2)  &  \sigma(x_1)        &  p\sigma^3(x_2) \\
          p^{-2} x_2   &     p^{-2}  \sigma^2(x_1) &  \sigma(x_2)        &  \sigma^3(x_1) \\
           y_1         &     p \sigma^2(y_2)       &  p^{-2} \sigma(y_1) & p^{-1} \sigma^3(y_2) \\
           y_2         &       \sigma^2(y_1)       &  p^{-2} \sigma(y_2) & p^{-2}  \sigma^3(y_1) \\
         \end{array}
       \right)
$$
where $x_1,x_2,y_1,y_2 \in \OO_L$ and $x_1 \neq 0$ or $x_2 \neq 0$ and $y_1 \neq 0$ or $y_2 \neq 0$.

\section{Conclusion }
We  gave a  general construction for MIMO MAC codes satisfying the generalized rank criteria and gave a lower bound for their decay.
As a corollary we got that all our codes  do achieve the pigeon hole bound.

\section*{Acknowledgement}
 The research of  T. Ernvall is supported in part  by the Academy of Finland grant 131745.
The research of  R. Vehkalahti is supported  by the Academy of Finland  grants 131745 and 252457.
The authors would like to thank Jyrki Lahtonen for suggesting this problem.

%\bibliographystyle{IEEE}	
%\bibliography{myrefs_special}	

\begin{thebibliography}{10}


\bibitem{Tse}
D.~Tse, P.~Viswanath, and L.~Zheng, ``Diversity and multiplexing tradeoff in
  multiple-access channels,'' \emph{IEEE Trans. Inf. Theory}, vol.~50, no.~9,
  pp. 1859--1874, 2004.

\bibitem{CoGaBo}    P. Coronel, M. G\"{a}rtner, and H. B\"{o}lcskei,   Selective-fading multiple-access MIMO channels, ``Diversity-multiplexing tradeoff and dominant outage event regions'', preprint available at http://www.nari.ee.ethz.ch/commth/pubs/p/CGB09 .


\bibitem{decaypaperi} J. Lahtonen, R. Vehkalahti, H.-F. Lu, C. Hollanti, and E. Viterbo, ``On the Decay of the Determinants of Multiuser MIMO Lattice Codes'', \emph{Proc. 2010 IEEE Inf. Theory Workshop}, Cairo, Egypt, Jan. 2010.

\bibitem{matriisi} H.-F. Lu, R. Vehkalahti, C. Hollanti, J. Lahtonen, Y. Hong, and E. Viterbo, ``New Space-Time Code Constructions for Two-User Multiple Access Channels'', \emph{IEEE J. Selected Topics in Signal Processing: Managing Complexity in Multiuser MIMO System}, vol. 3, no. 6, pp. 939-957, Dec. 2009.

\bibitem{LuHoVeLa} H. F. Lu, C. Hollanti, R. Vehkalahti, J. Lahtonen,  ``DMT Optimal Code Constructions for Multiple-Access MIMO Channel'', \emph{IEEE Trans. Inform. Theory}, vol. 57, no. 6, pp. 3594--3617, Jun. 2011.

\bibitem{Lu} H. F. Lu, ``Constructions of multi-block space-time coding schemes that achieve the diversity-multiplexing tradeoff'', \emph{IEEE Trans. Inform. Theory}, vol. 54, no. 8, pp. 3790--3796, Aug. 2008.

\bibitem{remarks} H.-F. Lu, J. Lahtonen, R. Vehkalahti, and C. Hollanti, ``Remarks on the criteria of constructing MAC-DMT optimal codes'', \emph{Proc. 2010 IEEE Inf. Theory Workshop}, Cairo, Egypt, Jan. 2010.

\bibitem{shidlovskii} Andrei B. Shidlovskii, \emph{Transcendental numbers}, Studies in mathematics, vol. 12, Walter de Gruyter, Berlin, New York, 1989.

%\bibitem{narkiewicz} W. Narkiewicz,\emph{Elementary and Analytic Theory of Algebraic Numbers}, Springer Monographs in Mathematics, Third Edition, 2004.

%\bibitem{washington} L. C. Washington, \emph{Introduction to Cyclotomic Fields}. Springer Graduate Texts in Mathematics 83, 1982.

\bibitem{bb} M. Badr and J.C. Belfiore, ``Distributed space-time block codes for the non-cooperative multiple-access channel'', \emph{Proc. 2008 International Zurich Seminar on Communication}, Zurich, Germany, Mar. 2008, pp. 132-135.

%\bibitem{ono} T. Ono, \emph{An Introduction to Algebraic Number Theory}, The University Series in Mathematics, Plenum Press, New York and London, 1990.

\end{thebibliography}

\end{document}